\documentclass[a4paper]{amsart}  
\usepackage{fullpage}
\usepackage{pb-diagram}
\usepackage{graphicx}
\usepackage{amssymb,amsmath}

\DeclareSymbolFont{script}{U}{eus}{m}{n}
\DeclareMathSymbol{\Wedge}{0}{script}{"5E}

\newtheorem{theorem}{Theorem}
\newtheorem{cor}{Corollary}
\newtheorem{definition}{Definition}
\newtheorem{remark}{Remark}
\newtheorem{proposition}{Proposition}
\newtheorem{lemma}{Lemma}

\newcommand{\acknowledge}{\subsection*{Acknowledgements}}

\DeclareMathOperator{\id}{id}

\DeclareMathOperator{\ri}{r}

\DeclareMathOperator{\End}{End}

\DeclareMathOperator{\spanu}{span}

\DeclareMathOperator{\Ph}{ \Phi_t^V}
\DeclareMathOperator{\Phe}{ \Phi_t^{V^c}}

\newcommand{\iso}{\text{iso}}
\newcommand{\aut}{\text{aut}}
\newcommand{\s}{\mathfrak{s}}
\newcommand{\R}{\mathbb{R}}
\renewcommand{\H}{\mathbb{H}}
\renewcommand{\u}{\mathfrak{u}}
\newcommand{\m}{\mathfrak{m}}
\newcommand{\so}{\mathfrak{so}}
\renewcommand{\sp}{\mathfrak{sp}}
\newcommand{\h}{\mathfrak{h}}
\begin{document}

\title{ C-projective symmetries of submanifolds in quaternionic geometry}


\author{Aleksandra Bor\'owka$^\dag$ and Henrik Winther$^\ddag$}
\address{$\dag$ Institute of Mathematics, Jagiellonian University 30-348 Krak\'ow,  Poland}
\email{aleksandra.borowka@uj.edu.pl}
\address{$\ddag$ Department of Mathematics and Statistics, Masaryk University, Kotlarska 2, 611 37 Brno, Czechia.}
\email{winther@math.muni.cz}

\maketitle
\begin{abstract}
The generalized Feix--Kaledin construction shows that c-projective $2n$-manifolds with curvature of type $(1,1)$ are precisely the submanifolds of quaternionic $4n$-manifolds which are fixed points set of a special type of quaternionic $S^1$ action $v$. In this paper, we consider this construction in the presence of infinitesimal symmetries of the two geometries. First, we prove that the submaximally symmetric c-projective model with type $(1,1)$ curvature is a submanifold of a submaximally symmetric quaternionic model, and show how this fits into the construction. We give conditions for when the c-projective symmetries extend from the fixed points set of $v$ to quaternionic symmetries, and we study the quaternionic symmetries of the Calabi-- and Eguchi-Hanson hyperk\"ahler structures, showing that in some cases all quaternionic symmetries are obtained in this way.

\end{abstract}
\section{Introduction}
Let $(S,J)$ be a complex manifold equipped with a complex connection $\nabla$. A curve $\gamma$ is called {\it $J$--planar} if it satisfies
$$ \nabla_{\gamma^\cdot}\gamma^\cdot \in \text{Span}_\R(\gamma^\cdot, J \gamma^\cdot) $$
The {\it c--projective class} $[\nabla]_c$ of $\nabla$ is then the class of complex connections which share the same $J$--planar curves, up to reparameterization. The triple $(S,J,[\nabla]_c)$ is called a c--projective manifold, and the assignment of a class $[\nabla]_c$ to a complex manifold is called a c--projective structure. 
An almost quaternionic structure on a manifold $M$ of dimension $4n$ is a smooth subbundle $Q\subset \text{End}(TM)$ which is of rank 3, and which can be locally generated by almost complex structures $I,J,K$ which anti-commute amongst themselves. A connection which preserves $Q$ is called a {\it quaternionic connection}. If $Q$ admits a torsion--free quaternionic connection, then it is called quaternionic, dropping the word 'almost'. From here on we assume that all quaternionic connections are torsion--free.
Quaternionic-- and c--projective geometries are both examples of parabolic geometries\cite{CS}. In both cases, the main differential invariant is a kind of Weyl curvature, and in the c--projective case, this curvature decomposes into 3 parts, called the $(2,0), (1,1)$ and $(0,2)$--parts.\cite{Alex}\cite{CEMN}

The first author and D. Calderbank established a relationship between c-projective geometries with Weyl curvature of type $(1,1)$ and quaternionic geometries in \cite{BC}. It turns out that real analytic c-projective manifolds with Weyl curvature of type $(1,1)$ are precisely the maximal totally complex submanifolds of quaternionic manifolds with a local circle action of a special kind. Moreover, the family of quaternionic manifolds containing a fixed c-projective submanifold $S$ with type $(1,1)$ c-projective curvature is  parametrized by holomorphic line bundles equipped with compatible real-analytic complex connections on $S$ with type $(1,1)$ curvature. From this point of view it is natural to ask what the consequences are for the algebras of c-projective and quaternionic symmetries, and in particular for the dimension of the algebra of submaximal symmetries.

For parabolic Cartan geometries on manifolds, the algebra of infinitesimal symmetries, i.e., the Lie algebra of vector fields whose local flow preserves the structure (see for example \cite{CS}, \cite{Zal}) is finite dimensional. The maximal dimension of such an algebra, for fixed dimension of the underlying manifold, is achieved for flat structures and usually, if the structure is not flat, the possible symmetry dimension drops significantly. Therefore it is natural to ask for the maximal dimension of the algebra of symmetries for non-flat structures, and this is called the submaximal symmetry. When the drop between maximal and submaximal symmetry dimension is significant, it is called a gap phenomenon, and for parabolic Cartan geometries in general this was studied by B. Kruglikov and D. The in \cite{KT}. The dimension of the submaximal symmetry algebra for $4n$-dimensional almost quaternionic structures has been computed by the second author, B. Kruglikov and L. Zalabova in \cite{KWZ}, and for $2n$-dimensional almost c-projective structures by B. Kruglikov, V. Matveev and D. The in \cite{KMT} (both for $n>1$). In these papers the authors also consider the submaximal dimensions of sub-cases of the corresponding structures depending on the type of the Weyl curvature.

From \cite{BC} we know that the flat, hence maximally symmetric, model of quaternionic geometry can be constructed from the flat model of c--projective geometry. 
The current paper is motivated by the following two observations. First, we discovered that the model of submaximal quaternionic manifolds constructed in \cite{KWZ} admits an $S^1$ action of the type required in \cite{BC}, and that the restriction of the (class of quaternionic connections of the) quaternionic structure to the underlying maximal totally complex submanifold gives the c-projective type $(1,1)$ submaximally symmetric model (see \cite{KMT}).
This generalizes the flat case to show that at least one submaximal quaternionic model can be constructed from a submaximal c--projective model.
Secondly, we observed that the dimension of the symmetry algebra in the case of submaximal c-projective type $(1,1)$ is equal to $2n^2-2n+4$, while in the quaternionic case it is $2(2n^2-2n+4)+1$, where $n$ is the complex and quaternionic dimension respectively. This suggested that it may be possible to describe the algebra of quaternionic symmetries in terms of c-projective symmetries on a totally complex submanifold.

It turns out that any vector field preserving the initial data used in generalized Feix--Kaledin construction can be carried through the twistor construction. Therefore we obtain the following theorem. 

\begin{it}
Let $V$ be a symmetry of a real-analytic c-projective manifold $S$ with c-projective curvature of type $(1,1)$, and suppose that $V$ preserves a connection on a holomorphic line bundle on $S$ (associated with the holomorphic tangent bundle) with type $(1,1)$ curvature. Then $V$ extends from the submanifold $S$ to a quaternionic symmetry on the quaternionic manifold obtained by a generalized Feix--Kaledin construction from these data. 
\end{it}

For maximal totally complex submanifolds of quaternionic manifolds, there is a well defined normal subbundle of a tangent bundle along submanifolds given by the action of any anti-commuting (with the complex structure on the submanifold) almost complex structure. 
Because of the maximal and submaximal examples, we expected that at least in the hypercomplex case, any symmetry from the submanifold will give an `orthogonal' quaternionic symmetry which would explain the relation between c-projective and quaternionic submaximal symmetry dimension. However this is not true in general, as if this was the case then for $n=2$ the Calabi metric on the cotangent bundle of $\mathbb{CP}^n$ would have the dimension of quaternionic symmetries equal at least to the submaximal quaternionic dimension. This is not possible, as we prove that the Calabi metric on the cotangent bundle of $\mathbb{CP}^n$ has quaternionic symmetries dimension which is neither maximal nor submaximal. In fact by direct computations (using the DifferentialGeometry package in Maple) we show that for the case $n=1$ (which is the Eguchi-Hanson metric) the quaternionic symmetry algebra  is generated by the $S^1$ action and the symmetries extended from the c-projective symmetries on the submanifold, so we obtain no `orthogonal' symmetries in that case.

The structure of the paper is as follows. In Section \ref{back} we provide necessary background about c-projective and quaternionic structures and generalized Feix--Kaledin construction \cite{BC}. Additionally we recall some material about flows and complexification that will be necessary for our considerations. In Section \ref{model} we discuss the properties and relationship between the submaximal models of c-projective and quaternionic structures from \cite{KMT} and \cite{KWZ}. Moreover we show how the submaximally symmetric quaternionic model arises from the submaximally symmetric c-projective type $(1,1)$ model by the generalized Feix--Kaledin construction, which generalizes the corresponding statement for the maximally symmetric model. In Section \ref{c-p} we show how c-projective symmetries on the maximal totally complex submanifold of quaternionic manifold $M$, which is the fixed-point set of the quaternionic $S^1$ action with no tri-holomorphic points, extend to quaternionic symmetries on $M$. In Section \ref{quat} we further discuss the properties of the algebra of quaternionic symmetries of a hypercomplex manifold which admit the $S^1$ action with the conditions required in \cite{BC}. We prove that any submaximally symmetric quaternionic manifold admits such an action. Finally we study the quaternionic symmetries of the Calabi and Eguchi-Hanson structure. We end with conclusions and further directions.

\section{Background}\label{back}
 \subsection{Projective Cartan geometries and their symmetries}
 Recall that a projective structure on a manifold is an equivalence class of  torsion--free connections which have the same unparametrized geodesics. 
 A classical result shows that the connections in the projective class are parametrized by $1$-forms:
 $$D\sim_p D' \Leftrightarrow \exists_{\gamma}:\quad D'_YZ=D_YZ +\gamma(Y)Z+\gamma(Z)Y,$$
where $\gamma$ is a $1$-form.

C-planar curves are complex analogues of geodesics. These are curves which for a given connection $D$ satisfy the equation 
$$D_{\dot{c}}\dot{c}\in\spanu\{\dot{c},J(\dot{c})\},$$
where $J$ is an almost complex structure on a $2n$-manifold $S$. For the purpose of this paper we assume that $J$ is integrable. We define a c-projective class of connections on $(S,J)$ as those torsion--free complex connections ($DJ=0$) which have the same c-planar curves. Similarly to projective structures, the connections within the class are parametrised by $1$-forms in the following way:
 $$D\sim_c D' \Leftrightarrow \exists_{\gamma}:\quad D'_YZ=D_YZ +\frac{1}{2}(\gamma (Y)Z+\gamma (Z)Y-\gamma(JY)JZ-\gamma(JZ)JY).$$
 
 An almost quaternionic structure on a manifold is a rank $3$ subbundle $Q\subset \End (TM)$ which is locally generated by three anti-commuting almost complex structures $I,J,K$ satisfying the quaternionic relation $IJ=K$. We call an almost quaternionic structure integrable (or simply quaternionic) if there exists a torsion free connection which preserves the subbundle $Q$. In that case, we can define a quaternionic analogue of geodesics, called q-planar curves, which  satisfy the equation 
 $$D_{\dot{c}}\dot{c}\in\spanu\{\dot{c},Q(\dot{c})\}.$$
 It turns out that the class of torsion free connections which have the same q-planar curves coincides with the class of all (torsion-free) quaternionic connections (i.e.,$ DQ\subset Q$) and is parametrized by $1$-forms in the following way:
  $$D\sim_q D' \Leftrightarrow \exists_{\gamma}:\quad D'_YZ=D_YZ +\frac{1}{2}(\gamma (Y)Z+\gamma (Z)Y-\Sigma_{i=1}^3(\gamma(I_iY)I_iZ+\gamma(I_iZ)I_iY))$$ where $I_1,I_2,I_3$ is a pointwise frame (see \cite{Alex}).

Hypercomplex manifolds are a special case of quaternionic manifolds, i.e. the quaternionic manifolds for which $Q$ is trivial and generated by integrable sections $I,J,K$. Then there exists a preferred connection in the class of quaternionic connections, called the Obata connection, with the property
$$DI=DJ=DK=0.$$
Such a connection is unique for a given hypercomplex structure. Note however, that a quaternionic structure may admit more than one non-equivalent hypercomplex structure. If a hypercomplex manifold admits a metric which is K\"ahler with respect to $I$,$J$ and $K$, then the manifold is called hyperk\"ahler.

Projective, c-projective and quaternionic structures are examples of parabolic Cartan geometries with abelian nilradical. For such geometries one can consider an underlying Weyl (or harmonic) curvature which in our cases is just the part of the curvature of the connections which is invariant under a change of connections in the class. One can show that for any connection, the remaining part of the curvature comes from the Ricci tensor (which for each case is normalized differently), therefore if the class of connections admits a Ricci flat connection, then the Weyl curvature coincides with the curvature of the connection (see \cite{Alex} for the quaternionic case). 

Parabolic geometries have finite dimensional algebras of infinitesimal symmetries (i.e. of vector fields with flows that preserve the structure) which, for fixed manifold dimension, is of maximal dimension for the corresponding flat structure. Therefore it makes sense to study the maximal dimension (for fixed dimension of the underlying manifold) of the algebra of infinitesimal symmetries in the non-flat case, which is called the submaximal dimension. It is usually much smaller than the dimension in the flat case, and this observation is called the gap phenomenon for parabolic geometries. The general theory was provided in \cite{KT} and the specific cases of c-projective and quaternionic geometries have been studied in \cite{KMT} and \cite{KWZ}, and we will now briefly recall the results obtained there.

\begin{definition}
Let $(S,J,[D]_c)$ be a c-projective $2n$-manifold. A vector field $V$ is an infinitesimal c-projective symmetry if the Lie derivatives along $V$ of $J$ and of the class $[D]_c$ vanish, or equivalently if $[\Ph ]_* J=J$ and for any $D\in[D]_c$ there exists $D'\in[D]_c$ such that $[\Ph] ^*D=D'$, where $\Ph$ is the local flow of $V$.
\end{definition}

If the c-projective structure is minimal (which is related to the notion of minimal torsion and always holds in the torsion-free case which we are interested in) then the Weyl curvature decomposes into three pieces: type I, type II and type III, where type I is the $(2,0)$ part of the tensor, type II is the $(1,1)$ part and type III does not appear in the torsion-free case. A necessary condition for the minimal c-projective structure to be submaximal is that its curvature is supported purely in one of these types. Note that the Feix--Kaledin construction deals only with manifolds with Weyl curvature of type II. 
In \cite{KMT}, an explicit model of non-flat c-projective structure with type II Weyl curvature which admits the (sub--) maximal dimension of infinitesimal c-projective symmetries is given for each dimension greater than $2n=2$. The symmetry dimension is $2n^2-2n+4$, and this turns out to be equal to the submaximal dimension for general c-projective structures, except for the case $2n=4$. Moreover, the submaximal type II structure turns out to be unique and we will discuss some of its properties in detail in Section \ref{model}.

\begin{definition}
Let $(M,Q,[D]_q)$ be a quaternionic manifold.  A vector field $V$ is an infinitesimal quaternionic symmetry if the Lie derivative preserves the space of (local) sections $\Gamma Q$, or equivalently if for any (local) almost complex structure $I_i\in \Gamma Q$ there exists a (local) almost complex structure $I_j\in \Gamma Q$ such that $[\Ph ]_* I_i=I_j$. In that case, the flow of $V$ also preserves the class of quaternionic connections  i.e., for any $D\in[D]_q$ there exists $D'\in[D]_q$ such that $[\Ph] ^*D=D'$.
\end{definition}
Unlike in the c-projective case, for almost quaternionic structures we do not need to deal with minimal and non-minimal cases. This is because the quaternionic structure is defined by the subbundle $Q$, and the class of minimal almost quaternionic connections (which for quaternionic structures is the class of torsion--free quaternionic connections) is determined (see \cite{Alex}). The Weyl curvature decomposes into two pieces, one of which corresponds to the torsion part, and therefore will not be considered here. The submaximal dimension for quaternionic  $4n$ manifolds is equal to $4n^2-4n+9$ for $n>1$, and a locally hypercomplex submaximal model is given in \cite{KWZ}. We will discuss the properties of this model in Section \ref{model}. Note, however, that it is not known if the submaximal quaternionic model is unique.
\subsection{Tractor calculus}\label{trac}
Tractor calculus is an important tool in theory of parabolic Cartan geometries (\cite{CEMN},\cite{CS} ). In this section we will briefly recall the results that will be needed for our considerations. 
\begin{definition}\label{o}
Let $(S,J)$ be a complex manifold of real dimension $2n$. Then locally we define the holomorphic line bundle $\mathcal{O}(1)$ on $S$ as the $(n+1)$-st root of  the bundle $\Lambda^nT^{1,0}S$. If $k=\frac{p}{q}$ is a rational number then by $\mathcal{O}(k)$ we denote $p$-th power of the $q(n+1)$-st root the bundle $\Lambda^nT^{1,0}S$ and by $\mathcal{O}(-k)$ we denote its dual.
\end{definition}
\begin{remark}
Locally the bundle from Definition \ref{o} always exists. If $(S,J)$ is the complex projective n-space, then $\mathcal{O}(1)$ exists globally and is equal to the dual to the tautological line bundle, hence the notation is compatible with the standard notation for line bundles over $\mathbb{CP}^n$ in algebraic geometry.
\end{remark}
On the bundle $j^1\mathcal{O}(1)$ of $1$-jets of $\mathcal{O}(1)$ there are linear connections associated to holomorphic projective and c-projective Cartan connections. This is because $j^1\mathcal{O}(1)$  is the co-tractor bundle for holomorphic projective and c-projective Cartan geometries and the underlying linear connections are co-tractor connections. The co-tractor connections have the important property that they encode the properties of the Cartan connections - in fact they can be used as an alternative description of parabolic Cartan geometries. As a consequence, the Weyl curvature of the c-projective structure is of type $(1,1)$ if and only if the curvature of the c-projective co-tractor connection is of type $(1,1)$. Also the co-tractor connection is flat if and only if the underlying structure is flat (i.e., its Weyl curvature vanishes). A choice of connection $D$ in a c-projective class (or in a projective class respectively) gives a splitting of the $1$-jet sequence. Then $j^1\mathcal{O}(1)\cong_D \mathcal{O}(1)\oplus T^{1,0}S\otimes \mathcal{O}(1)$, and the explicit formulas for c-projective and holomorphic projective co-tractor connections is   
$$\mathcal{D}^D_Y\left(\begin{array}{c}l\\\alpha\end{array}\right)_D=\left(\begin{array}{c}D_Yl-\alpha(Y)\\ D_Y\alpha+(\ri^D)_Yl\end{array}\right)_D,$$
where $\ri^D$ is a particular normalization of the Ricci tensor, different for projective and c-projective structures. We denote by $ \ri_p^D$ the normalized Ricci tensor in the projective case and by $\ri_c^D$ the normalized Ricci tensor in the c-projective case.

\subsection{Twistor geometry of quaternionic manifolds}\label{twist}
Let $\mathcal{Q}\subset Q$ be the fibre bundle over an almost quaternionic manifold $M$ consisting of those endomorphisms that square to $-\id$. If locally $I,J,K$ are three anti-commuting almost complex structures satisfying the quaternionic relations 
and generating $Q$ then $\mathcal{Q}=\{a_1I+a_2J+a_3K;\ (a_1,a_2,a_3)\in S^2\}$, hence the fibres of $\mathcal{Q}$ are $2$-spheres or equivalently complex projective lines. The total space of $\mathcal{Q}$ is called the twistor space of $M$ and is denoted by $Z$ (see \cite{Sal}). It carries a naturally induced almost complex structure which is integrable precisely when $M$ is quaternionic. The fibres of $\mathcal{Q}$ are holomorphic projective lines in $Z$, called twistor lines, with normal bundle isomorphic to $\mathcal{O}(1)\otimes\mathbb{C}^{2n}$. $Z$ also carries a real structure $\tau$, which for fixed fibre $x$ of $Q$ maps an almost complex structure at $x$ to its opposite. By definition the real structure is invariant on the fibres of $\mathcal{Q}$, and its restriction to any of them is equal to the antipodal map on $\mathbb{CP}^1$. 
It turns out that there is a one-to-one correspondence between quaternionic manifolds and their twistor spaces:
\begin{theorem}[\cite{PP}]
Let $Z$ be a holomorphic manifold of complex dimension $2n+1$ equipped with a real structure $\tau$. Suppose, moreover, that $Z$ contains a smooth projective line with a normal bundle isomorphic to  $\mathcal{O}(1)\otimes\mathbb{C}^{2n}$ which is $\tau$-invariant and on which $\tau$ does not have any fixed point. Then the space of all $\tau$-invariant lines with normal bundle $\mathcal{O}(1)\otimes\mathbb{C}^{2n}$ on which $\tau$ has no fixed points  (called real twistor lines) is a quaternionic manifold of real dimension $4n$ such that $Z$ is locally isomorphic to its twistor space.
\end{theorem}
It is possible to read off properties of quaternionic manifolds from the properties of their twistor spaces. In particular quaternionic vector fields (symmetries) correspond to holomorphic vector fields on the twistor space which are $\tau$-invariant and transversal to real twistor lines. 

If $Z$ is the twistor space of a hypercomplex manifold then it additionally admits a holomorphic projection to $\mathbb{CP}^1$. The fibres of the projection are given by the distribution induced on $Z$ (as a bundle over $M$) by the Obata connection.

 \subsection{Flows and complexifications}\label{comp}
 In this section we will recall some facts regarding flows and complexifications. Although the results are classical, in case of complexifications, they are very technical so for reference one may see \cite{Biel}, \cite{BC}.
 
 Let $V$ be a vector field on a manifold $S$ and  $\Phi_t^V$ its local flow. For a fixed $x\in S$ we can find a neighbourhood $U_x$ of $x$ and $\epsilon$ such that  $\Phi_t^V$ is defined for any $|t|<\epsilon$. Then for any $|t|<\epsilon$ we have that  $\Phi_t^V|_{U_x}$ is an isomorphism onto its image. As a consequence for any $|t|<\epsilon$ it gives an isomorphism of $T_xS$ and $T_{\Phi_t^V(x)}S$ (this is really just the push--forward and therefore we denote it by $(\Phi_t^V)_*)$. Using this we can define a natural flow on $TS$ and on any associated bundle, with the property that its bundle projection is equal to $\Phi_t^V$, by the formula:
 $$\Phi_t^{V^c}(x,v)=(\Phi_t^V(x),(\Phi_t^V)_*(v)).$$
 
Now recall that a complexification $S^c$ of a real-analytic $n$-manifold $S$ is a complex manifold $S^c$ of a real dimension $2n$ admitting a real structure $\theta$ (i.e. an anti-holomorphic involution) whose fixed points set is analytically diffeomorphic to $S$. One obtains complexifications of real analytic manifolds by using holomorphic extensions of real analytic coordinates. If we assume that $S$ admits an integrable complex structure $J$, then $S^c$ carries two transverse foliations (defined by the fact that the subbundles $T^{1,0}S$ and $T^{0,1}S$ of the complexified tangent bundle are integrable) by holomorphic leaves which are biholomorphic to $(S,J)$ and $(S,-J)$ respectively. We denote them $(1,0)$ and $(0,1)$-foliations. In this case there is a natural model for $S^c$ equal to the product $S\times \overline{S}$, where the real structure is the map $(x,y)\mapsto (y,x)$ and the fixed points set $S_{\mathbb{R}}$ is the diagonal. 

If $\mathcal{L}$ is a holomorphic rank $k$ bundle over $S$ then we can also consider a complexification $\mathcal{L}^c$ over $S^c$ which is a rank $2k$ bundle. Then $\mathcal{L}^c=\mathcal{L}^{1,0}\oplus\mathcal{L}^{0,1}$, where $\mathcal{L}^{1,0}$ is trivial along the leaves of the $(0,1)$ foliation (and hence it is a pull back of a bundle over $S$ - which happens to be $\mathcal{L}$) and $\mathcal{L}^{1,0}$ is trivial along the leaves of the $(1,0)$ foliation (and it is a pull back of $\overline{\mathcal{L}}$ over $\overline{S}$).

If $D$ is a connection on $S$ then (using holomorphic extensions of connection forms) we can extend it as a holomorphic connection to $S^c$, but only locally near $S_{\mathbb{R}}$ - so if we use the model $S\times\overline{S}$, the connections from $S$ extend only in a neighbourhood of the diagonal. As a consequence, if $S$ is equipped with a real-analytic c-projective structure, then $S^c$ near $S_{\mathbb{R}}$ is equipped with a holomorphic  c-projective structure (which is the structure defined analogously to a c-projective structure, but with the requirement that all connections are holomorphic).

Let $\nabla$ be a real-analytic connection on $\mathcal{L}$ which is compatible with the holomorphic structure and let $\nabla^c$ be its complexification using holomorphic extensions. Then $\nabla^c$ induces connections on $\mathcal{L}^{1,0}$ and $\mathcal{L}^{1,0}$ which are trivial along the leaves of the $(0,1)$ and $(1,0)$ foliations respectively.

The last object that we will need to complexify is a real-analytic vector field on $S$. Again using real analytic extensions we can extend this to a holomorphic vector field on $S^c$. Note that by definition this vector field preserves the two foliations on $S^c$ coming from the complexification.
 \subsection{The Generalized Feix--Kaledin construction} \label{FK}
 One of the first nontrivial examples of hyperk\"ahler metrics is Calabi metric on the cotangent bundle of the complex projective $n$-space. This motivated the question of existence of hyperk\"ahler metrics on the cotangent bundle of a K\"ahler manifold. Using twistor methods, B. Feix in \cite{Feix} provided a construction of such a metric on a neighbourhood of the zero section of the cotangent bundle.
 Generalizations of this construction were given by Feix for the hypercomplex case, where the hypercomplex structure is constructed on a neighbourhood of the zero section of the tangent bundle of a real-analytic complex manifold with a real-analytic complex connection with type $(1,1)$ curvature (see \cite{Feix2}), and by the first author and D. Calderbank for the quaternionic case (see \cite{BC}). We call the construction from \cite{BC} the generalized Feix--Kaledin construction. In all cases, the obtained manifolds admit a quaternionic $S^1$ action such that its fixed point set is the initial complex manifold. One can show that the $S^1$ action has an additional property: it has no triholomorphic points, which means that for any point $x$, the natural action that it induces on $\mathfrak{sp}(1)\subset Q_x\subset \mathfrak{gl}(T_xM)$ is non-trivial. Conversely, any hypercomplex or quaternionic manifold admitting such an $S^1$ action  can be obtained in a neighbourhood of the fixed points set by the constructions, provided that the fixed point set has maximal dimension (i.e. its dimension is  equal to the half of the dimension of hypercomplex/quaternionic manifold). Now we will give the details of the generalized Feix--Kaledin construction, as this will be necessary for our considerations. 
Let $(S,J,[D]_c)$ be a real analytic complex $2n$-manifold with a real analytic torsion-free c-projective structure $[D]_c$ such that its c-projective Weyl curvature is of type $(1,1)$. Moreover, let $\mathcal{L}$ be a holomorphic line bundle on $S$ with a real analytic connection $\nabla$ which is compatible with the holomorphic structure and with curvature of type $(1,1)$.

   The construction can be divided into the following steps.
 \begin{itemize}
 \item We complexify $S$ and all structures on $S$. As we discussed in Section \ref{comp}, $S^c$ carries two holomorphic transverse foliations such that the leaves are biholomorphic to $S=(S,J)$ and $\overline{S}=(S,-J)$ respectively. One can show that along leaves, the complexified c-projective structure gives flat projective structures. We also complexify $\nabla$, obtaining connections on  $\mathcal{L}^{1,0}$ and $\mathcal{L}^{1,0}$ which are trivial along the leaves of the $(0,1)$ and $(1,0)$-foliations respectively and flat along  the leaves of the $(1,0)$ and $(0,1)$-foliations.
 \item We define two holomorphic rank $n+1$ bundles over the leaf spaces $S^{0,1}\cong \overline{S}$ and $S^{1,0}\cong S$ of the $(1,0)$ and $(0,1)$-foliations by requiring that the fibre over a leaf $x$ is equal to the space of flat sections of the tensor product connection of the projective co-tractor connection (see Section \ref{trac}) on $x$ and the connection induced by $\nabla^c$ on $\mathcal{L}^{1,0}$ and $\mathcal{L}^{0,1}$ respectively. We call these the  bundles of affine sections of $\mathcal{L}^{1,0}(1)$ and $\mathcal{L}^{0,1}(1)$ respectively (as the parallel sections of the corresponding twisted co-tractor bundles correspond to special sections of  $\mathcal{L}^{1,0}(1)$ and $\mathcal{L}^{0,1}(1)$) and denote them by $\mathcal{A}^{1,0}$ and $\mathcal{A}^{0,1}$ respectively. We denote 
$$\mathcal{V}^{1,0}:=  [\mathcal{A}^{1,0}\otimes\mathcal{L}^{0,1}]^*,\ \ \ \ \ \ \ \mathcal{V}^{0,1}:=  [\mathcal{A}^{0,1}\otimes\mathcal{L}^{1,0}]^*.$$
 \item We consider evaluation maps $\phi_{1,0}:\ \mathcal{L}^{0,1}(1)\otimes [\mathcal{L}^{1,0}(1)]^*\rightarrow \mathcal{V}^{1,0}$ defined by 
 $$(x,y,l)\mapsto [y, s\mapsto s(x)\otimes l]$$ and $\phi_{0,1}:\ \mathcal{L}^{1,0}(1)\otimes [\mathcal{L}^{0,1}(1)]^*\rightarrow \mathcal{V}^{0,1}$ defined by 
 $$(x,y,l)\mapsto [x, s\mapsto s(y)\otimes l],$$ which are isomorphisms away from the zero sections.
 \item We observe that  $\mathcal{L}^{0,1}(1)\otimes [\mathcal{L}^{1,0}(1)]^*$ and $\mathcal{L}^{1,0}(1)\otimes [\mathcal{L}^{0,1}(1)]^*$ embed as complementary affine subbundles into the projective bundle $\mathbb{P}(\mathcal{L}^c)$, and use this along with the maps $\phi_{1,0}$ and  $\phi_{0,1}$ to glue the total spaces of $\mathcal{V}^{1,0}$ and $\mathcal{V}^{0,1}$. The resulting manifold is not Hausdorff.
\item The projective bundle  $\mathbb{P}(\mathcal{L}^c)$ admits a natural real structure given by complexification:  By the properties of complexification, for any $z\in S^c$  the fibre  $\mathcal{L}^{0,1}(1)\otimes [\mathcal{L}^{1,0}(1)]^*_z$ is naturally isomorphic to $\overline{\mathcal{L}^{0,1}(1)\otimes [\mathcal{L}^{1,0}(1)]^*}_{\theta{z}}$. If we compose this isomorphism with $-\id$ we obtain a real structure $\tau$ on $\mathbb{P}(\mathcal{L}^c)$ which does not have fixed points and is invariant on the fibres over $S_{\mathbb{R}}$. On the other hand, the natural real structure coming from complexification induces an anti-holomorphic isomorphism between  $\mathcal{V}^{1,0}$ and  $\mathcal{V}^{0,1}$. After composition with $-id$, this is compatible (via the maps $\phi_{1,0}$ and  $\phi_{0,1}$) with $\tau$ and hence we obtain a real structure on the gluing of $\mathcal{V}^{1,0}$ and $\mathcal{V}^{0,1}$, which also does not have fixed points. We denote it also by $\tau$.
 \item We choose open submanifolds of total spaces of $\mathcal{V}^{1,0}$ and $\mathcal{V}^{0,1}$ such that their union contains  the image of $\mathbb{P}(\mathcal{L}^c)|_{S_{\mathbb{R}}}$, is invariant under the real structure and is Hausdorff. Hence we obtain a holomorphic manifold $Z$ of dimension $2n+1$ with a real structure which has no fixed points.
 \item It turns out that the normal bundle to the images of fibres of $\mathbb{P}(\mathcal{L}^c)|_{S_{\mathbb{R}}}$ is isomorphic to $\mathcal{O}(1)\otimes\mathbb{C}^{2n}$. Since the images of the fibres are projective lines, which are invariant under the real structure, it follows that they are twistor lines and $Z$ is a twistor space of a quaternionic manifold $M$ (note that the image of an arbitrary fibre of  $\mathbb{P}(\mathcal{L}^c)$ is also a twistor line but not real).
 \item As $M$ is the space of all real twistor lines in $Z$, it contains the space $S=S_{\mathbb{R}}$ of real twistor lines given by images of fibres of $\mathbb{P}(\mathcal{L}^c)|_{S_{\mathbb{R}}}$. One can show that $S$ is totally complex in $M$ and that the quaternionic structure induces the c-projective structure $[D_c]$ on $S$.
 \item The twistor space $Z$ admits a holomorphic $\mathbb{C}^*$ action given by scalar multiplication (and its inverse) in  $\mathcal{V}^{1,0}$ and $\mathcal{V}^{0,1}$ which is transversal to images of fibres of $\mathbb{P}(\mathcal{L}^c)|_{S_{\mathbb{R}}}$. As for  $S^1\subset \mathbb{C}^*$, the action maps real twistor lines to real twistor lines, so it follows that it corresponds to a quaternionic $S^1$ action on $M$ on some neighbourhood of $S$ and moreover, we have that $S$ is its fixed points set.
 \item If we set $\mathcal{L}=\mathcal{O}(-1)$ with $\nabla$ induced by some connection $D\in [D]_c$ with curvature of type $(1,1)$, then the construction coincides with the standard Feix construction \cite{Feix}, hence $M$ is hypercomplex (but note that the $S^1$ action is not, as it is not triholomorphic). 
 
  \end{itemize}

Let us now summarize the results obtained in \cite{BC} concerning the construction described above.
 \begin{theorem}\cite{BC}
The space $Z$ constructed by the generalized Feix--Kaledin construction is the twistor space of a quaternionic $4n$-manifold $M$ such that $M$ admits a quaternionic $S^1$ action which has no triholomorphic points, and whose fixed point set is $S\subset M$. Moreover, the c-projective structure on $S$ used in the construction is the restriction of the quaternionic connections of the quaternionic structure on $M$. Conversely, any quaternionic $4n$-manifold admitting a quaternionic $S^1$ action with no triholomorphic points and with fixed point set of dimension $2n$ can be obtained locally near $S$ by the generalized Feix--Kaledin construction for some ($\mathcal{L},\nabla$).
\end{theorem} 
 
 \section{A Motivating Example}\label{model}
 In \cite{BC} it was shown that the flat quaternionic structure can be constructed by the standard Feix construction from the flat c-projective structure. These happen to be the standard structures on either quaternionic-- or complex projective space, and they both realize maximal symmetry dimension in their respective categories.
The fact that a maximal model can be constructed from a maximal model is the first sign that the Feix construction, and more generally the generalized Feix--Kaledin construction, is affected by the presence of infinitesimal symmetries of the involved geometric structures.

 In this section we will generalize this result by showing that one may similarly construct a submaximally symmetric quaternionic structure from the submaximally symmetric C-projective structure with curvature of type $(1,1)$, for some line bundle and connection.
We begin by studying the properties of the submaximal type 2 c-projective structure from \cite{KMT}. Recall that such a structure (for each $n>1$) is unique and in local coordinates $z^i, \overline{z^i}$ on $\mathbb{C}^n$ it is given by (the c-projective class of) a complex connection with the only non-vanishing Christoffel symbols:
$\Gamma^2_{11}=\overline{z^1}$ and $\Gamma^{\overline{2}}_{\overline{1}\overline{1}}$.

 \begin{proposition}\label{cprojconunique}
The submaximal c-projective model of curvature type (1,1) admits a unique invariant connection $D_0$.  
\end{proposition}
\begin{proof}
	The submaximal c-projective model is shown to be unique in \cite{KMT}, and explicit symmetry vector fields are given in section 3 of that paper: the generators are the real and imaginary parts of
\begin{align*}
&\partial_{z_1}-\tfrac{1}{2}\bar z_1^2 \,\partial_{\bar z_2}\,,\,\partial_{z_2},\ldots,\partial_{z_n}\\
&z_1 \partial_{z_1} + 2 z_2 \partial_{z_2}+ \bar z_2 \partial_{\bar z_2}, \, z_i \partial_{z_j}, \, (i\not=2,j>1)
\end{align*}
	From there we can see that the symmetry algebra realizes the model as a reductive klein geometry. Let $\h$ be the isotropy algebra at $x=0$, generated by the last vector fields. Then by Nomizu's theorem, invariant connections are in bijective correspondence with $\h$--equivariant bilinear maps $\m \rightarrow \m^\ast \otimes \m$, where $\m$ is a reductive complement to $\h$. Thus the statement that the invariant connection is unique is equivalent to the statement 
	\begin{eqnarray}
		(\m^\ast \otimes \m^\ast \otimes \m)^\h = \{ 0 \}
		\label{}
	\end{eqnarray}
Thus we may compute isotropy invariant tensors rather than dealing with connections. It is clear that the vector fields with indices 1,2 form a subalgebra, while the vector fields involving only indices $3,\ldots, n$ form a complex affine algebra. These subalgebras commute, and lets us decompose $\m = \m_{1,2} \oplus \m_{3,\ldots,n}$. The complex affine algebra is large enough that there is obviously no invariant elements in any tensor products involving $\m_{3,\ldots,n}$, with the only possible exception being elements of $\m_{1,2}\otimes Id_{3,\ldots,n}$ or similar, but these fail to be invariant because there is no trivial submodule in $\m_{1,2}$.  Thus any invariant tensor would be in $(\m_{1,2}^\ast \otimes \m_{1,2}^\ast \otimes \m_{1,2})^\h$. Verifying that the latter space is equal to $\{ 0\}$ is a simple linear algebra computation in fixed dimension. This means that there is a unique invariant connection, corresponding to the tensor $0$. 
\end{proof}

\begin{remark}
The invariant connection corresponding to the Nomizu map $0$ on a reductive Klein geometry is also known as the canonical connection in the literature.
\end{remark}
 
Next we consider the submaximally symmetric quaternionic structure $M$ from \cite{KWZ}, with symmetry algebra $\mathfrak{g}$ as described in Section $5$ of that paper.  
 This quaternionic structure is described as the functional span of local almost complex structures $I,J,K$. This frame is locally hypercomplex, and so there exists a unique Obata connection $D$, i.e. a torsion-free connection which satisfies $D I = D J = D K =0$. The isotropy algebra $\h$ is given there as
$$
\h=\sp(1)\oplus\bigl(\frak{so}(2)\oplus\R\oplus\mathfrak{gl}(n-2,\H)\bigr)\ltimes
\mathfrak{heis}(8n-12,\H)
$$
and the isotropy representation is given as quaternionic matrices acting on $\m\simeq \H^n$. The $\h$--ideal $\sp(1)$ acts as a linear quaternionic structure on $\m$. We may then find a diagonal operator in the other ideal, in particular contained in $\so(2)\oplus \sp(n-2)$, where $\sp(n-2)\subset \mathfrak{gl}(n-2,\H)$, which acts as a linear almost complex structure. Call this $v_{\so(2)\oplus \sp(n-2)}$. Since the diagonal of $\so(2)\oplus \sp(n-2)$ is given by blocks of the form
$$
\begin{bmatrix}
0 & -1 & 0 & 0 \\
1 & 0 & 0 & 0 \\
0 & 0 & 0 & 1 \\
0 & 0 & -1 & 0 \\
\end{bmatrix}
$$
in a real basis respecting the quaternionic structure, we may find a block diagonal operator from $\sp(1)$ consisting of blocks of the form
$$
\begin{bmatrix}
0 & -1 & 0 & 0 \\
1 & 0 & 0 & 0 \\
0 & 0 & 0 & -1 \\
0 & 0 & 1 & 0 \\
\end{bmatrix}.
$$
Call the latter element $v_{\sp(1)}$. Then the element
$$
v = v_{\sp(1)} + v_{\so(2)\oplus \sp(n-2)}.
$$	
admits an eigenspace with eigenvalue $0$ of dimension $2n$. Let $h_1,h_2,\ldots,h_{4n}$ be local coordinates as in Section 5.1 of \cite{KWZ}. Then we may write the vector field corresponding to $v$ as
\begin{eqnarray}
v = h_4\partial_{h_3}-h_3 \partial_{h_4} +\ldots + h_{4n}\partial_{h_{4n-1}}-h_{h_{4n-1}} \partial_{h_{4n}}
	\label{b}
\end{eqnarray}
It is easy to see that the vanishing set of $v$ is given by
\begin{eqnarray}
	\{x\, | \, h_j(x) = 0,\, j\mod 4 \equiv 3,4 \}
	\label{}
\end{eqnarray}
This yields a local submanifold $S$ of dimension $2n$, with coordinates $h_j, j \mod 4 \equiv 1,2$. By construction, $v$ acts non-trivially on the quaternionic bundle restricted to $S$, fixing a sub-bundle of rank 1, and acting as non-trivial rotations on the fibres of the complement. 

The above observations can be summarized in the following remark.
\begin{remark}
The submaximal quaternionic model $M$ from \cite{KWZ} admits an $S^1$ action with fixed points $S$ of dimesnion $2n$ which locally near $S$ has no triholomorphic points. Therefore, $M$ arises locally from $S$ by the generalized Feix--Kaledin construction, where the c-projective structure on $S$ is induced (by restriction and projection) by the quaternionic structure on $M$.
\end{remark}

Recall that,  by the generalized Feix--Kaledin construction, from a fixed c-projective structure (of right type) on $S$ we can construct many quaternionic manifolds by using different holomorphic line bundles with a connection (of right type) on $S$ (see Section \ref{FK}). Therefore for full understanding of the relation between $S$ and $M$ we need to verify what is the line bundle in this case. Because (see \cite{KWZ}) is locally hypercomplex, we can expect that in this case the construction is the hypercomplex Feix construction and therefore the line bundle is $\mathcal{O}(-1)$ with a connection induced by some real-analytic connection in the c-projective class with type $(1,1)$ full curvature. This turns out to be true which we will discuss in the remaining part of this section.

\begin{proposition}
	The normalizer algebra $N(v)$ of $v$ restricts to, and acts locally transitively on $S$. The dimension of $N(v)$ is $2n^2-2n+5$. 
\end{proposition}
\begin{proof}
	The normalizer algebra $N(v)$ preserves the spaces of invariants of $v$. This means that all vector fields in $N(v)$ are tangent to $S$. Since the symmetry algebra of $Q$ is a reductive Klein geometry, we have an invariant complement $\m$ to the isotropy $\h$, and the normalizer of $v$ can be split in the same way. Thus $N(v)\cap \m$ is the zero-eigenspace of $v$ in $\m$, of dimension $2n$. This yields local transitivity on $S$. Write again
$$
v = v_{\sp(1)} + v_{\so(2)\oplus \sp(n-2)}.
$$	
The subalgebra $\sp(1)$ is an ideal in $\h$, which gives the contribution $\langle v_{\sp(1)}\rangle $  only to the normalizer. For the rest of $\h$, the isotropy representation is given by quaternionic matrices. In this representation, the element $v_{\so(2)\oplus \sp(n-2)}$ is a linear almost complex structure. The normalizer of this component will then be precisely those elements whose entries are actually complex numbers. Adding all of this together yields the desired dimension of $N(v)$.
\end{proof}
\begin{proposition}
	The Obata connection $D$ restricts to $S$, i.e. if $\xi,\zeta \in T_xS$ then $D_\xi \zeta \in T_xS$. The restricted connection $D_0 = D|_S$ is invariant with respect to $N(v)$.
\end{proposition}
\begin{proof}
	A similar computation to the one in the proof of Proposition \ref{cprojconunique} shows that the Obata connection is the unique invariant connection on $M$. Therefore both connections are represented by the Nomizu map $0$ in their respective Lie algebras. The normalizer $N(v)$ a is a sub-algebra of the quaternionic symmetries, and is also a reductive Klein geometry in itself, so the connection represented by $0$ on $S$ must be the restriction of the one represented by $0$ on $M$.
\end{proof}
\begin{proposition}
	$S$ is a c-projective manifold, and its c-projective symmetry algebra is the restriction of the algebra $N(v)$ to $S$, of dimension $2n^2-2n+4$, which is the sub-maximal symmetry dimension.
\end{proposition}
\begin{proof}
	The $v$ invariant subbundle of $Q|_S$ admits an $N(v)$ invariant section $J$, which is an almost-complex structure on $S$. Since the Obata connection is torsion free and $J$ is parallel, $J$ is a complex structure. The c-projective class is generated by the connection $D_0 = D|_S$. The symmetry dimension is 1 less than $\dim N(v)$, because $v$ itself acts trivially by definition.
\end{proof}

By explicit computations we can also show that the Lie derivative $\mathcal{L}_vJ=0$ and that there is a hypercomplex frame $I,J,K$ such that $\mathcal{L}_vI=K$ and $\mathcal{L}_vK=I$. Therefore we conclude that $M$ indeed arise from $S$ by the hypercomplex Feix--Kaledin construction for some $D'$ from c-projective class. We also know that in this case the restriction of the Obata connection to $S$ is equal to $D'$, hence $D'=D_0$.

We can summarize the results in the following theorem.
 \begin{theorem}
 Let $(S,J,[D])$ be the $2n$-dimensional c-projective manifold with c-projective curvature of type $(1,1)$ for which the dimension of c-projective symmetries is submaximal and let $D$ be the unique connection in the c-projective class which is preserved by all c-projective symmetries. Then the manifold obtained by standard Feix--Kaledin construction from $(S,J,D)$ (or equivalently by generalized Feix--Kaledin construction from $(S,J,[D])$, $(\mathcal{O}(-1),\nabla^D)$) is the submaximal (torsion-free) quaternionic model discussed in \cite{KWZ}.
 
 \end{theorem}
 \section{C-projective symmetries}\label{c-p}
It is easy to show (see \cite{BC}) that totally complex submanifolds of quaternionic manifolds admit induced real analytic c-projective structures such that their c-projective curvature is of type $(1,1)$. In this section we study which c-projective symmetries extend to quaternionic symmetries from maximal totally complex submanifolds $S\subset M$ in the case when they are the zero set of a quaternionic $S^1$ action with no tri-holomorphic points, i.e., when it arises by the generalized Feix--Kaledin construction (see Section \ref{FK}). 
 
Let $(S,J,[D]_c)$ be a real analytic c-projective manifold of real dimension $2n$ $(n>1)$ with a c-projective curvature of type $(1,1)$. Let $V$ be a symmetry of the c-projective structure on $S$ and denote by $\Phi_t^V$ the local flow of $V$. 
\begin{lemma}\label{pD}
The local flow  $\Phi_t^V$ preserves the complex structure and the tractor connection on the co-standard tractor bundle.
\end{lemma} 

\begin{proof}
Recall that the co-standard tractor bundle $\mathcal{T}^*$ is the $1$-jet bundle $J^1\mathcal{O}(1)$, and a choice of a connection $D$ in the c-projective class gives a splitting of the $1$-jet sequence such that $\mathcal{T}^*=\mathcal{O}(1)\oplus [T^*S\otimes\mathcal{O}(1)]$ (see Section \ref{trac}). The complex structure on $\mathcal{T}^*$ is induced from $J$ on $S$ in the natural way hence clearly $\Ph$ preserves the complex structure on $\mathcal{T}^*$. As $\Ph$ preserves the c-projective structure, at any point for any connection $D$, there exists a connection in the c-projective class such that $(\Ph)_*(D)=D'$. As a consequence, if we apply $(\Ph)^{*}$ to the $1$-jet sequence we get that the direct sum decomposition of $\mathcal{T}^*$ given by $D$ is transformed into direct sum decomposition given by $D'$, i.e., $$(\Ph)_*\left(\begin{array}{c}l\\\alpha\end{array}\right)_D=\left(\begin{array}{c}(\Ph)^*l\\(\Ph)_*\alpha\end{array}\right)_{D'}.$$ Recall that choosing the connection $D$, the explicit formula for the tractor connection is $$\mathcal{D}^D_Y\left(\begin{array}{c}l\\\alpha\end{array}\right)_D=\left(\begin{array}{c}D_Yl-\alpha(Y)\\ D_Y\alpha+(\ri_c^D)_Yl\end{array}\right)_D.$$
We have
$$(\Ph)_*\mathcal{D}[(\Ph)_*\left(\begin{array}{c}l\\\alpha\end{array}\right)_D]=(\Ph)_*(\mathcal{D}^D\left(\begin{array}{c}l\\\alpha\end{array}\right)_D)=(\Ph)_*\left(\begin{array}{c}Dl-\alpha(\cdot)\\ D\alpha+(\ri_c^D)_{(\cdot)}l\end{array}\right)_D=$$
$$=\left(\begin{array}{c}(\Ph)_*[Dl-\alpha(\cdot)]\\  (\Ph)_*[ D\alpha+(\ri_c^D)_{(\cdot)}l]\end{array}\right)_{D'}=\left(\begin{array}{c}D'(\Ph)^*l-(\Ph)_*\alpha(\cdot)\\ D'(\Ph)_*\alpha+(\ri_c^{D'})_{(\cdot)}(\Ph)^*l\end{array}\right)_{D'}=$$  $$=\mathcal{D}^{D'}[\left(\begin{array}{c}(\Ph)^*l\\(\Ph)_*\alpha\end{array}\right)_{D'}]=\mathcal{D}[(\Ph)_*\left(\begin{array}{c}l\\\alpha\end{array}\right)_D].$$
Hence $(\Ph)_*\mathcal{D}=\mathcal{D}$ which finishes the proof.
\end{proof}
\begin{cor}
Suppose that $\mathcal{L}$ is a holomorphic line bundle associated to a holomorphic tangent bundle on $S$ and $\nabla$ a complex connection on $S$ compatible with the holomorphic structure with curvature of type $(1,1)$. Moreover, suppose that $V$ is also a symmetry of $(\mathcal{L},\nabla)$. Then the local flow $\Ph$ also preserves the twisted tractor connection which is a tensor product connection on $\mathcal{T}^*\otimes \mathcal{L}$.
\end{cor}

Let us consider a complexified co-tractor bundle over $S^c$ which (as it is complex, see Section \ref{comp}) splits into a direct sum of pull-backs of bundles over $S$ and $\overline{S}$. Observe that the bundles are isomorphic to projective co-tractor bundles on $S$ and $\overline{S}$, and the complexified tractor connection induces connections on them. Recall that $S^c$ admits two transverse foliations by holomorpic leaves (see Section \ref{comp}). It follows from using the properties of the complexifications and the results from \cite{BC} that the obtained connections are trivial along the leaves in one direction (different for the two connections).
Moreover, along the leaves in the other direction they are projective co-tractor connections. The conditions on curvatures imply that the projective structures on the leaves are flat (hence projective tractor connections along the leaves are flat, see Sections \ref{trac} and \ref{FK}) and connections on the line bundles are also flat along the leaves. 
\begin{lemma}
The symmetry vector field $V$ extends to a vector field on a complexification $S^c$, and it preserves the complexified twisted tractor connection as well as the product structure (hence it preserves the corresponding connections on the pull back bundles). 
\end{lemma}
\begin{proof}
Recall that the complexification of a complex manifold with coordinates $z_i$, $\overline{z_i}$ can be defined by replacing conjugated coordinates $\overline{z_i}$ by independent complex coordinates $\tilde{z_i}$, and the real submanifold is then defined by the requirement $\tilde{z_i}=\overline{z_i}$. In this setting $S^c$ is locally a product $S\times\overline{S}$. Then we can define a local flow on $S^c$ (for simplicity also denoted by $\Ph$), by using the flow on $S$ (and hence on $\overline{S}$) and the product structure. By definition, this flow preserves the product structure on any bundle arising from a complexification. Moreover, by the properties of complexifications of connections and vector bundles, the flow defined in this way preserves the complexified tractor connection.
\end{proof}
 \subsection{Extending symmetries from c-projective manifolds}
 
Recall that any vector field on a manifold defines a vector field on its tangent bundle and on any associated bundle in a natural way (see section \ref{comp}). Denote by $V^c$ the natural vector field defined by a complexification of $V$ (as a holomorphic extension to $S^c$) on the complexified tractor bundle. 
As it is induced by a vector field that is a complexification itself,  $V^c$ preserves the product decomposition of the complexified tractor bundle, hence it induces a vector field on the components. 
Recall that in the generalized Feix--Kaledin construction we consider bundles of so called affine sections (see Section \ref{FK}). Now we will show that $V^c$ defines vector fields on the total spaces of these bundles.
 \begin{proposition}\label{Flowaff}
For small time $t$, the local flow of $V^c$ maps affine sections along leaves into affine sections. Hence $V^c$ descends to a vector field on the total space of the bundles of affine sections along leaves.
 \end{proposition}
  \begin{proof}
Let us fix a leaf $x_0$ of the $(0,1)$ foliation and an affine section $l(x_0,\tilde{x})$ of $\mathcal{L}^{1,0}$ along this leaf. By definition of $V^c$ we know that $\Phe [\left(\begin{array}{c}l\\ Dl\end{array}\right)_D]$ is a section of the leaf $\Ph (x_0)$ for $t$ small enough (we may also have to restrict a leaf to some open subset). 
By the definition of the push-forward we have 
$$(\Ph)_*\mathcal{D}(\Phe [\left(\begin{array}{c}l\\ Dl\end{array}\right)_D])=\mathcal{D}[\left(\begin{array}{c}l\\ Dl\end{array}\right)_D].$$ 
 Therefore, as by Lemma \ref{pD}, $(\Ph)_*\mathcal{D}=\mathcal{D}$ then along the leaf $\Ph (x_0)$
$$ \mathcal{D}\Phe[\left(\begin{array}{c}l\\ Dl\end{array}\right)_D]=0,$$ which implies that the flow descends to the bundle of affine sections which finishes the proof.
 \end{proof}
Recall that in the generalized Feix--Kaledin construction we construct the twistor space of a quaternionic manifold as a gluing of bundles $\mathcal{V}^{1,0}$ and $\mathcal{V}^{0,1}$ which are dualisations of bundles of affine sections of $\mathcal{L}^{1,0}$ and $\mathcal{L}^{0,1}$ along the corresponding leaves tensored by $\mathcal{L}^{0,1}$ and $\mathcal{L}^{1,0}$ (see Section \ref{FK}). Therefore the above proposition shows that the vector fields induced by $V^c$ on the complexified tractor bundle, on  $\mathcal{L}^{1,0}$ and on $\mathcal{L}^{0,1}$, define vector fields on $\mathcal{V}^{1,0}$ and $\mathcal{V}^{0,1}$.
 \begin{definition}
We denote by $ V^c_{1,0}$ the natural vector field on $\mathcal{V}^{1,0}$ and by  $ V^c_{0,1}$ the natural vector field on $\mathcal{V}^{0,1}$ defined by  $V^c$.
To simplify the notation we will denote the flows of natural vector fields on $\mathcal{V}^{1,0}$,   $\mathcal{V}^{0,1}$ and on $\mathbb{P}((\mathcal{L}^c)^*)$  by $\Phe$.
 \end{definition}
 
 \begin{proposition}
 $ V^c_{1,0}$ and $V^c_{0,1}$ glue to a vector field $\mathcal{V}$ on the twistor space.
 \end{proposition}
\begin{proof}

First we will show that on the image of $\phi$ (see Section \ref{FK}), the vector fields $V^c_{1,0}$ and $V^c_{0,1}$ are equal to the push forward of the natural vector field induced by $V$ on $\mathbb{P}((\mathcal{L}^c)^*)$ by the isomorphism $\phi$ . To see this let us fix $s\in  \mathcal{V}^{1,0}$ and suppose that $s=\phi_{1,0}(x_0,\tilde{x}_0, l_0\otimes\tilde{l}_0)$. We will show that the image by $\phi_{1,0}$ of the local flow $\Phe (x_0,\tilde{x}_0, l_0\otimes\tilde{l}_0)$ is the local flow $\Phe (s)$. Recall that $\phi_{1,0}(x,\tilde{x}, l\otimes\tilde{l})$ is a functional on the space of affine sections on the leaf $\tilde{z}=\tilde{x}_0$ (where $z_i,\tilde{z}_j $ are coordinates on $S^c=S\times\overline{S}$)  which maps affine sections $f$ into $f(x_0)\otimes l\otimes\tilde{l}\in\mathbb{C}$. By definition, the flow $s_t:=\Phe (s)$ is the curve of functionals such that for any affine function $f$ on the leaf $\tilde{x}_0$ we have that $s_t(f_t)=s(f)$, where $f_t$ is the flow of $f$. As we have seen in Proposition \ref{Flowaff}, this is well defined as for fixed $t$ all $f_t$ are affine functions on the leaf $\tilde{x}_t$. Now observe that also by definition $f_t(x_t)$ is the flow of the natural vector field on $\mathcal{L}^{1,0}$ induced by $V^c$. This proves the claim as if $l_t$ is the natural flow on $\mathcal{L}^{1,0}$ induced by $V^c$ and $r_t$  is the flow of the natural vector field on $(\mathcal{L}^{1,0})^*$ induced by $V^c$ with $r_0=\frac{a}{l_0}$ for some $a\in\mathbb{C}$ then $t_t\otimes r_t=\underline{a}$.
\end{proof}
Recall that we assume that $n>1$.
 \begin{theorem}\label{csym}
 Let $V$ be a c-projective symmetry on a $2n$-dimensional c-projective manifold $(S,J,[D]_c)$, and suppose that $V$ also preserves a connection $\nabla$ on a bundle $\mathcal{L}$ associated to the tangent bundle. Then $V$ extends from the submanifold $S$ to a quaternionic symmetry $\mathcal{V}$ on the quaternionic manifold $M$ obtained by the generalized Feix--Kaledin construction from $(S,J,[D]_c, \mathcal{L},\nabla)$.
 \end{theorem}
 \begin{proof}
 By construction $\mathcal{V}$ is a holomorphic vector field on the twistor space of $M$ which is transversal to twistor lines that are close enough to images of fibers of $\mathbb{P}((\mathcal{L}^c)^*)$. It is also invariant under the real structure, as both the real structure and the vector field come from complexification (in a natural way). Hence it corresponds to a quaternionic symmetry on $M$ (see section \ref{twist}).
 
By the definition of  $\mathcal{V}|_S$, the induced vector field on $M$ coincides with $V$ along $S$.
 \end{proof}
 \begin{cor}
The algebra of quaternionic symmetries of the quaternionic manifold $M$ obtained by the generalized Feix--Kaledin construction from $(S,J,[D]_c, \mathcal{L},\nabla)$ contains a subalgebra consisting of those c-projective symmetries that preserve $\nabla$.
 
In particular, if $(\mathcal{L},\nabla)$ is trivial then the algebra of local quaternionic symmetries of the quaternionic manifold obtained from generalized Feix--Kaledin construction admits a subalgebra isomorphic to the c-projective symmetry algebra on the submanifold. 
 \end{cor}
 
 Let $(\mathcal{L},\nabla)=(\mathcal{O}(-1),\nabla^D)$, where $\nabla^D$ is the connection induced from a real analytic connection $D$ in the c-projective class with type $(1,1)$ curvature. Recall that in this case the generalized Feix--Kaledin construction is just the hypercomplex Feix--Kaledin construction for $(S,J,D)$.
 \begin{proposition}
 In the case of the hypercomplex Feix--Kaledin construction, all symmetries $\mathcal{V}$ obtained as extensions of c-projective symmetries preserving $D$ are hypercomplex.
 \end{proposition}
 \begin{proof}
As $\mathcal{L}^{0,1}(1)$ and $ \mathcal{L}^{1,0}(1)$ are trivial in this case, by definition the flows of natural vector fields induced by $V^c$ on them are constant lifts of flows of $V^c$. Also $\mathcal{V}^{1,0}$ and $\mathcal{V}^{0,1}$ are bundles dual to bundles of affine functions, and the flows on them are constant on constant functions. Since the hypercomplex projection on the twistor space is given by evaluation on constant sections, it follows that the flow of $\mathcal{V}$ is tangent to the Obata distribution (i.e to the fibers of the projection).
 \end{proof}
 \begin{remark}\label{rem4}
 In the case when $S$ is equipped with the flat c-projective structure and $(\mathcal{L},\nabla)=(\mathcal{O}(-1),\nabla^D)$ for $D$ flat, the obtained hypercomplex manifold is locally equivalent to flat quaternionic structure which can also be obtained from $S$ using trivial $(\mathcal{L},\nabla)$. Therefore in this case the requirement that $V$ preserves $D$ for extending c-projective symmetries is not necessary. Note, however, that in this case it is a necessary and sufficient condition for extending c-projective symmetries to hypercomplex symmetries. 
 
 In general we are unable to extend all c-projective symmetries (i.e., the ones that do not preserve $\nabla$) when $\nabla$ is not trivial. For example in the case of grassmannian $Gr_2(\mathbb{C}^{4})$, the dimension of all quaternionic symmetries is $15$ (and one of them is the $S^1$ action, whereas it is constructed from $S$ equipped with the flat c-projective structure (and ($\mathcal{L}^{-\frac{1}{2}},$ $\nabla$) induced by the Fubini-Study metric), which has dimension of c-projective symmetries equal to $16$. We expect that the flatness of $\nabla$ is the only obstruction for the condition that $V$ preserves $\nabla$ to be necessary.
 
 Note that the $S^1$ action provides an additional symmetry on $M$ which is independent from the symmetries obtained as extensions of c-projective symmetries from $S$ (as $S$ is the fixed points set of the action).
 
 Finally we would like to note that, although we proved the result for $n>1$, the generalized Feix--Kaledin construction works also for case $n=1$ and in this case the structure that we have on $S$ is a conformal Cartan connection (c-projective structure on $S$ is not enough data). As the methods we used to prove Theorem \ref{csym} are also based on Cartan connections, the proof will extend to the case $n=1$. We skip the details as they are very technical, and also because studies of submaximal symmetry dimension in \cite{KWZ} was done only for $n>1$. 
 \end{remark}

 \section{Quaternionic symmetries}\label{quat}
 
  In this section we further discuss the properties of the algebra of quaternionic symmetries of some manifolds arising by the generalized Feix--Kaledin construction. 
We start by showing that any submaximally symmetric quaternionic manifold arise locally in this way.
\subsection{Submaximal quaternionic structures}
In this subsection we consider $n>1$.
\begin{proposition}
	Let $(M,Q)$ be a sub-maximal quaternionic structure of dimension $4n$ and $x\in M$ a regular point. Then the isotropy algebra $\iso_x$ contains a subalgebra $K$ isomorphic to $\u(2)\simeq \u(1)\oplus\sp(1)$ for which all elements act with purely imaginary spectrum on the isotropy representation $\m$.  
	Moreover there exists an element $V\in K$ such that $V$ acts trivially on a subspace $\s\subset \m$ with $\dim \s = 2n$
\end{proposition}
\begin{proof}
	The proof is by inspection of the isotropy subalgebra given in \cite{KWZ} and its canonical representation. Let $V=\zeta+a\xi$, where $\zeta \in \u(1)\subset \so(2)\oplus \sp(n-2)$ is non-zero with the same eigenvalues in each component, and  $\xi\in\sp(1)$ be non-zero. A matrix computation gives that there exist non-zero $a\in \R$ such that the spectrum of $V$ contains 0 with multiplicity $2n$. Fix this $a$, and let $\s$ be the kernel of $V$ in $\m$.
\end{proof}
\begin{remark}
The image of $V$ under the isotropy representation is the same as $v$ from section \ref{model}.
\end{remark}
\begin{proposition}
	The vector field generating $V$ vanishes along a local submanifold of dimension $2n$ near $x$.
\end{proposition}
\begin{proof}
	First, we show that there exists a subalgebra $s$ which projects to $\s$ at $x$ and which normalizes $V$. The isotropy sub-algebra $\so(2)\oplus \sp(1) \oplus \sp(n-2)$ admits a reductive decomposition of the symmetry algebra because it acts with imaginary eigenvalues. Thus we can find symmetry vector fields which are tangent to the $0$-eigenspace $\s$ of $V$ at $x$. These vector fields generate some sub-algebra $s$. Because the fields have highest weight $0$ with respect to $V$ and the Lie-bracket is $V$-equivariant, the algebra $s$ normalizes (even centralizes) $V$. This algebra projects to a distribution of rank at most $2n$ in a neighbourhood of $x$, due to highest weight, and thus there exists a neighbourhood of $x$ where the rank is exactly $2n$. 
	
	When we translate points along a symmetry, the isotropy changes by conjugation. Therefore the isotropy of points reachable by translations $\exp s$ contains $V$. This means precisely that the vector field generating $V$ vanishes along $\exp s \cdot x$, which due to the involutive and locally constant rank invariant distribution, must be a local submanifold $S$ of dimension $2n$.
\end{proof}
\begin{proposition}
	The vector field generating $V$ has the properties required to guarantee that $(M,Q)$ arises from the generalized Feix--Kaledin construction.
\end{proposition}
\begin{proof}
	For $y\in M$ we have the map $\iso_y\rightarrow \aut(Q_y)\simeq \sp(1)$. By construction, the image of $V$ under this map is nonzero for all $y$ in some neighbourhood $U$ of $x$. Since this image is 1-dimensional, it defines a local vector sub-bundle of $Q$, which has rank 1. Intersection of the sub-bundle with the sphere of pointwise square roots of $-Id$ in $Q$ yields two local almost complex structures $\pm J\in Q$, and these are invariant with respect to the vector field generating $V$.
\end{proof}

\subsection{Some examples}
 Recall that in Section \ref{model} we have shown than the submaximal quaternionic model from \cite{KWZ} arises from the submaximal c-projective model by the generalized Feix--Kaledi construction. In this case, the algebra of quaternionic symmetries has dimension $2(2n^2-2n+4)+1$ and contains a $(2n^2-2n+4)$-dimensional subalgebra of extensions of c-projective symmetries and, independent from them, the $S^1$ action. Hence we have $(2n^2-2n+4)$ additional independent quaternionic symmetries that do not arise in this way. Initially we conjectured that at least in the hypercomplex case, the manifold $M$ obtained by the generalized Feix--Kaledin construction from $(S,[D]_c,\mathcal{L},\nabla)$ has an algebra of quaternionic symmetries of dimension at least $2k+1$, where $k$ is the dimension of subalgebra of c-projective symmetries on $S$ that preserve $\nabla$. By studying the properties of the Calabi and Eguchi-Hanson metrics we will show that this is not true. In fact in the case of the Eguchi-Hanson metric the only quaternionic symmetries are the $S^1$ action and the ones that arise in a similar way as in Theorem \ref{csym}. Note that the Eguchi-Hanson metric is defined on a $4$-dimensional manifold, hence this is formally not included in Theorem \ref{csym}, but as explained in Remark \ref{rem4} the result can be extended.
\subsubsection{Symmetries of the Calabi quaternionic structure}
It is known (see \cite{DS}) that the Calabi hyperK\"ahler structure admits an isometry algebra which realizes it as a cohomogeneity 1 Riemannian manifold. This means that the associated quaternionic structure will also admit many symmetries, but to our knowledge the symmetries of the Calabi quaternionic structure alone have not been considered in the literature. Whether this structure happens to realize either maximal or submaximal symmetry dimension is a natural question. We will show that the answer is negative.

\begin{proposition}
	The Calabi quaternionic structure of dimension $4n$ does not realize the maximal symmetry dimension of $4(n+1)^2 -1$.
\end{proposition}
\begin{proof}
	If a quaternionic structure admits maximal symmetry dimension, then it is locally equivalent to the flat structure on $\H P^n$. Since this is not so, the symmetry dimension is strictly smaller. 
\end{proof}

\begin{proposition}
	Let $n>1$. The Calabi quaternionic structure of dimension $4n$ does not realize the sub-maximal symmetry dimension of $4n^2-4n+9$.
\end{proposition}
\begin{proof}
	If the symmetry algebra is of sub-maximal dimension, then its isotropy algebra at a regular point is isomorphic to the isotropy algebra given in \cite{KWZ}. This isotropy has a maximal compact subalgebra (meaning adjoint operators are semisimple with purely imaginary spectrum) isomorphic to $\so(2)\oplus \sp(1)\oplus \sp(n-2)$. The isometry algebra of the Calabi metric is $\mathfrak{su}(n+1)$, and the isotropy of its principal orbit is $\u(n-1)$ (see \cite{DS}). 	
	
	First, let $n>3$. We note that the Calabi isotropy does not admit an embedding into the (maximal compact subalgebra of the) sub-maximal isotropy algebra.
	
	Next, let $n=2$. Let $\mathfrak{k}$ be the maximal compact subalgebra of the symmetry subalgebra $\mathfrak{g}$, and $\mathfrak{k}_0=\so(2)\oplus \sp(1)$ be the isotropy part. 
	The adjoint representation $\mathfrak{g}$ is completely reducible with respect to $\mathfrak{k}$. If we restrict to $\mathfrak{k}_0$, the decomposition is
$$
\mathfrak{g} = \so(2)\oplus \sp(1) \oplus \R \oplus  \mathfrak{heis}(4) \oplus \H \oplus \H
$$
where each summand is irreducible, and the first 4 summands comprise the isotropy algebra of $\mathfrak{g}$.  Since the Calabi isometry acts by cohomogeneity 1, there exists a $\mathfrak{k}_0$--sub-module of dimension at least 7, which is included in $\mathfrak{k}$. From the decomposition, we see that this must be $\H \oplus \H$. But the 4-dimensional Abelian subalgebra $\mathfrak{heis}(4)$, which acts nilpotently and hence is non-compact, also acts non-trivially on $\H \oplus \H$. This contradicts complete reducibility with respect to $\mathfrak{k}$.
	
	Finally, the case $n=3$ can be shown by a modification of the case $n=2$.
	 All cases together yield that the Calabi quaternionic structure cannot be locally equivalent to a sub-maximal structure near a regular point for $n>1$.
\end{proof}

\begin{remark}
The Calabi  hyperk\"ahler structure arises by the standard Feix construction from the flat c-projective structure using the Fubini--Study connection. For $n=2$ the dimension of c-projective symmetries preserving this connection is $8$ which is equal to the submaximal c-projective symmetries dimension. We have shown that the  Calabi  quaternionic structure is neither maximally nor submaximally symmetric, which implies that its quaternionic symmetry dimension is strictly less than the submaximal dimension, which is $17=2\cdot 8+1$. 
\end{remark}

\subsubsection{Symmetries of the Eguchi-Hanson quaternionic structure}
The Eguchi-Hanson hyperK\"ahler structure appears as a special case of the Calabi construction for $n=1$, and in this case we are able to explicitly compute all quaternionic symmetries. In local coordinates the metric is given by
\begin{align*}
&g= \frac{\rho}{4(\rho^2-1)}d\rho^2+\frac{1}{\rho}((\rho^2-\cos(\psi)^2)d\phi^2 + \cos(\psi) \cos(\phi) \sin(\psi) \sin(\phi) d\phi d\psi \\
&-\sin(\psi)^2\sin(\phi)^2 \cos(\psi)d\phi d\theta +\sin(\phi)^2(\cos(\psi)^2\cos(\phi)^2+\rho^2-\cos(\phi)^2)d\psi^2\\
&+\sin(\phi)^3\sin(\psi)^3\cos(\phi)d\psi d\theta-\sin(\psi)^2 \sin(\phi)^2(\cos(\psi)^2\cos(\phi)^2-\rho^2+1-\cos(\psi)^2-\cos(\phi)^2)d\theta^2,\\
\end{align*}
where the products of tensors mean $ab = a \otimes b + b \otimes a$ except for when $a=b$, in which case we mean $a^2 = a \otimes a$. This metric admits an isometry algebra isomorphic to $\u(2)$, generated by the vector fields
{\scriptsize \begin{align*}
&v_1 = \cos(\psi) \partial_\phi-\frac{\sin(\psi) \cos(\phi)}{\sin(\phi)} \partial_\psi+\partial_\theta,\\
&v_2 = \sin(\psi) \cos(\theta) \partial_\phi+\frac{\sin(\phi) \sin(\psi)^2 (\cos(\theta) \cos(\psi) \cos(\phi)+\sin(\phi) \sin(\theta))}{\cos(\psi)^2 \cos(\phi)^2-\cos(\psi)^2-\cos(\phi)^2+1}\partial_\psi+\frac{\sin(\phi) \cos(\theta) \cos(\psi)-\sin(\theta) \cos(\phi)}{\sin(\phi) \sin(\psi)} \partial_\theta,\\
&v_3 = \sin(\psi) \sin(\theta) \partial_\phi-\frac{\sin(\phi) \sin(\psi)^2 (\sin(\phi) \cos(\theta)-\sin(\theta) \cos(\phi) \cos(\psi))}{\cos(\psi)^2 \cos(\phi)^2-\cos(\psi)^2-\cos(\phi)^2+1} \partial_\psi+\frac{\sin(\phi) \sin(\theta) \cos(\psi)+\cos(\theta) \cos(\phi)}{\sin(\phi) \sin(\psi)} \partial_\theta,\\
&v_4 =  \cos(\psi) \partial_\phi+\frac{\sin(\psi) \sin(\phi) \cos(\phi)}{\cos(\phi)^2-1} \partial_\psi-\frac{\cos(\psi)^2 \cos(\phi)^2-\cos(\psi)^2-\cos(\phi)^2+1}{\sin(\phi)^2 \sin(\psi)^2} \partial_\theta.
\end{align*}}
Here $v_1$ is the central and $v_2,v_3,v_4$ forms a basis of the semi-simple ideal. 
Since the metric is hyperK\"ahler and the isometries preserve the space of parallel two-forms, these isometries are also quaternionic symmetries. 
\begin{proposition}
The isometry algebra described above forms the full quaternionic symmetry algebra of the quaternionic structure associated to the parallel hypercomplex structure coming from the Eguchi-Hanson metric. Hence this structure has symmetry dimension 4.
\end{proposition}
\begin{proof}
We have already exhibited 4 independent symmetry vector fields, and we will show that the dimension of the quaternionic symmetry algebra is at most 4. The Lie symmetry equation is  
$$
\omega(\mathcal{L}_X I_j)= 0
$$
where $\omega \in \text{Ann}(Q)$ and $I_j,j=1,2,3$ is a local frame for the bundle of holonomy-trivial skew-symmetric linear maps on the tangent space, i.e. a local hypercomplex structure generating the same quaternionic structure as the Eguchi-Hanson hypercomplex structure. This is a system of linear first order PDE for the 4 coefficient functions of $X$, each depending on 4 variables,  and it can be written explicitly by choosing a basis of $\text{Ann}(Q)$. The dimension of the symmetry algebra is the same as the dimension of the solution space $\mathcal{S}$ of this system.

Let $F^{(1)}$ be the Lie equation system. The first derivatives are not all determined, but we have symbols
\begin{align*}
&g^{(1)}_1 = \{  F^{(1)} =0  \} \subset T^\ast M\\
&\dim g^{(1)}_1 = 7\\
&\dim g^{(1)}_0 = 4\\
\end{align*}
Thus we consider the first prolongation-projection, $F^{(2)}=DF^{(1)} = F^{(2)}_2\cup F^{(2)}_1$. This yields symbols
\begin{align*}
&g^{(2)}_2 = \{  F^{(2)}_2 =0  \} \subset S^2T^\ast M\\
&g^{(2)}_1 = \{  F^{(2)}_1 =0  \} \subset T^\ast M\\
&\dim g^{(2)}_2 = 4\\
&\dim g^{(2)}_1 = 7\\
&\dim g^{(2)}_0 = 4\\
\end{align*}
Which does not yield any bound on the solution space dimension. We prolong again, and obtain $F^{(3)}=D^2F^{(1)} = F^{(3)}_3\cup F^{(3)}_2\cup F^{(3)}_1$ and symbols
\begin{align*}
&g^{(3)}_k = \{  F^{(3)}_k =0  \} \subset S^kT^\ast M\\
&\dim g^{(3)}_3 = 0\\
&\dim g^{(3)}_2 = 4\\
&\dim g^{(3)}_1 = 4\\
&\dim g^{(3)}_0 = 4\\
\end{align*}
This yields our first bound, which is $\dim \mathcal{S} \le 4+4+4 = 12$. But the system is still not in involution, so we prolong one more time, and obtain $F^{(4)}=D^2F^{(1)} = F^{(4)}_4\cup F^{(4)}_3\cup F^{(4)}_2\cup F^{(4)}_1$, and the symbols
\begin{align*}
&g^{(4)}_k = \{  F^{(4)}_k =0  \} \subset S^kT^\ast M\\
&\dim g^{(4)}_4 = 0\\
&\dim g^{(4)}_3 = 0\\
&\dim g^{(4)}_2 = 0\\
&\dim g^{(4)}_1 = 1\\
&\dim g^{(4)}_0 = 3\\
\end{align*}
This yields our final bound, $\dim \mathcal{S} \le 3+1 = 4$, and so we are done.
\end{proof}

 \subsection{Summary and Further directions}\label{furth}
 
 Let $M$ be a quaternionic manifold arising from the generalized Feix--Kaledin construction with the data \\$(S,[D]_c,\mathcal{L},\nabla)$, where $\mathcal{L}$ is a line bundle associated to the tangent bundle on $S$. Let $k$ be the dimension of the algebra of c-projective symmetries on $S$ and $k'$ be the dimension of the subalgebra of symmetries that preserve $\nabla$.
 In this paper we have studied the lower bound for dimension of quaternionic symmetries on $M$ in terms of $k$ and $k'$. We have shown that it is equal to $k'+1$, and this is realizable in the case of the Eguchi-Hanson structure, while for the submaximal structure this dimension is far from the bound and equal to $2k'+1=2k+1$. It is natural to ask if this is the upper bound in the non-flat case (for the flat quaternionic structure the dimension is equal to $2k+3$). As we have shown that  any submaximally symmetric quaternionic manifold arises locally by the generalized Feix--Kaledin construction, and as the submaximally symmetric type $(1,1)$ model is unique, this would be a major step towards either proving that the submaximal quaternionic structure is unique, or classifying the submaximal models.

 \acknowledge We would like to thank David Calderbank, Boris Kruglikov and Lenka Zalabova for helpful discussions and comments. This work was partially supported by the Simons - Foundation grant 346300 and the Polish Government MNiSW 2015-2019 matching fund, and by the grant P201/12/G028 of the Grant Agency of the Czech Republic.

\bibliographystyle{amsalpha}

\providecommand{\bysame}{\leavevmode\hbox to3em{\hrulefill}\thinspace}
\providecommand{\MR}{\relax\ifhmode\unskip\space\fi MR }
\providecommand{\MRhref}[2]{%
  \href{http://www.ams.org/mathscinet-getitem?mr=#1}{#2}
}
\providecommand{\href}[2]{#2}

\end{document}